\documentclass[journal]{IEEEtran}
\usepackage{cite}
\usepackage{url}
\usepackage{empheq}
\usepackage{float}
\usepackage{dsfont}
\usepackage[hidelinks]{hyperref}
\usepackage[normalem]{ulem}
\usepackage{amsmath,amssymb,amsfonts}
\allowdisplaybreaks[4]
\usepackage{graphicx}
\usepackage{textcomp}
\usepackage{amsmath}
\usepackage{enumerate}
\usepackage{epstopdf}
\usepackage{array}
\usepackage{booktabs}
\usepackage{subfigure}
\usepackage{multirow}
\usepackage{soul, color}
\usepackage[usenames,dvipsnames]{xcolor}
\usepackage[latin1]{inputenc}
\usepackage{cases}

\usepackage{amsthm}
\newtheorem{theorem}{Theorem}

\soulregister\cite7 
\soulregister\citep7 
\soulregister\citet7 
\soulregister\ref7 
\soulregister\pageref7

\usepackage{bbding}

\usepackage{nomencl}
\makenomenclature
\usepackage{etoolbox}
\renewcommand\nomgroup[1]{%
  \item[\bfseries
  \ifstrequal{#1}{A}{Acronyms}{%
  \ifstrequal{#1}{S}{Symbols}{%
  \ifstrequal{#1}{U}{Units}{}}}%
]}

\usepackage{amsmath}
\allowdisplaybreaks[4]

\usepackage{algorithm}
\usepackage{algpseudocode}

\allowdisplaybreaks
\begin{document}

\title {Online Coordinated Scheduling of EV Battery Charging and Heating in Cold Environments}

\author{Xiaowei Wang, 
Yize Chen, \textit{Member, IEEE},
Yue Chen, \textit{Senior Member, IEEE}
\thanks{This work was supported in part by Natural Sciences and Engineering Research Council of Canada (NSERC).}
\thanks{Xiaowei Wang and Yue Chen are with the Department of Mechanical and Automation Engineering, The Chinese University of Hong Kong, Hong Kong, China. E-mail: xwwang@mae.cuhk.edu.hk, yuechen@mae.cuhk.edu.hk.}
\thanks{Yize Chen is with the Department of Electrical and Computer Engineering, University of Alberta, Edmonton, AB, Canada. Email: yize.chen@ualberta.ca.}
}

\maketitle

\begin{abstract}
Electric vehicle (EV) adoption in cold regions is hindered by degraded EV charging performance at low temperatures, which necessitates effective battery thermal management during charging.
Given the coupling of battery charging and heating dynamics, this paper explores the benefits of online coordinated charging and heating control, rather than performing them separately.
 Specifically, we first build queue models for both battery charging and thermal dynamics. Then, we formulate an optimization problem to minimize the system cost of a charging station, which allows us to coordinate charging and heating through maintaining queue stability. To solve the problem, we develop an online coordinated charging and heating control algorithm using Lyapunov optimization techniques. The proposed algorithm is prediction-free and has provable feasibility and optimality guarantees.
 Numerical experiments based on real-world data demonstrate the effectiveness and robustness of our proposed method compared to existing benchmarks.

\end{abstract}


\IEEEpeerreviewmaketitle

\section{Introduction}

\IEEEPARstart{E}{lectric} 
vehicle (EV) sales continue to grow globally to comply with the emission reduction targets. Among this increasing trend, more than half of vehicles sold in China in 2024 is EV, while Europe and the United States are also accelerating electrification, achieving a sales share of about 60\% and 20\% by 2030, respectively \cite{outlook_ev}.
One of the major challenges hindering further EV adoption is cold winter climate, as low operating temperatures adversely impact charging performance of EVs~\cite{10105947}. Indeed, low temperatures primarily affect EV adoption in two aspects: decreased driving range and limited charging rates \cite{esparza2025electric}. Cold temperature leads to higher battery internal resistance, reduced capacity and battery life, and thus degraded performance \cite{WARNER2024191}. EV driving range is also reduced significantly in cold conditions due to additional heating energy consumption \cite{9548276}. In \cite{powell2024impact}, it is shown that for every degree Celsius below the optimal temperature, the driving range declines by about 0.8\%. Such impacts can be particularly evident in regions subject to cold climates such as North America, Northeastern Asia, and North Europe.

Degraded EV charging performance is another concern in cold weather.
 EV fast charging rate can significantly decline under cold conditions \cite{MOTOAKI2018162}.
Reference \cite{tikka2021technical} reports that a Nissan Leaf shows a 25\% increase in charging time with a 33\% decrease in charged energy at -20$^\circ$C. 
Moreover, reduced driving range results in more frequent charging behaviors in cold climates due to an extra recharge need, which even increases the peak demand of power systems \cite{10105947}.


Many existing studies have concentrated on developing EV charging strategies for cost savings~\cite{9185026}, carbon emission control~\cite{EV_carbon},
flexibility provision~\cite{EV_fleixbility}, and profit maximization~\cite{11641645}.
However, efforts mentioned above assume EV battery operates under perfect temperature conditions and ignore both low-temperature scenarios and battery thermal characteristics. To address the adverse effects of low temperatures on EV charging, battery heating management is proposed to improve charging performance at low temperatures~\cite{LINDGREN201637}. It is shown that cars with battery heating capability are less affected given low temperatures~\cite{tikka2021technical}.
To achieve EV charging with battery thermal management, a straightforward approach is to decouple battery charging and heating into two separate control processes without considering the coupling between thermal and charging dynamics. For instance, the battery temperature can be regulated at a fixed setpoint, while a conventional EV charging control strategy is simultaneously implemented under this preferred temperature condition.
However, the above uncoordinated EV battery charging and heating fails to determine the optimal timing and power allocation for heating based on the battery charging states, leading to increased costs~\cite{ruan_temperature}. Therefore, in view of the coupled nature of charging and heating, a natural question to ask is: 
\emph{How to implement effective EV charging and heating control in a coordinated manner for cost savings?} Despite the potential benefits of coordinated charging and heating control, there is insufficient understanding on its design. 


To date, a few studies have explored EV charging under cold climates.
Power requirements of heating and fast charging of EV battery in cold temperatures are investigated in \cite{10005127}. A temperature-aware EV battery operation model is proposed to assess the battery performance in cold weather\cite{11152525}. However, both studies only consider the pre-heating step before charging events, neglecting the temporal coupling between heating and charging.
Reference \cite{tte_bayram} designs a distributed fast EV charging control method in cold weather. Nevertheless, this work assumes that the charging rate of EVs is constant in cold conditions, thus not applicable to scenarios (e.g., workplace charging) where EVs experience long charging durations, during which the battery temperature gradually decreases due to heat loss to the ambient environment.
By taking the coupling between heating and charging into account, in \cite{10027463}, an offline optimization problem is formulated to jointly solve battery thermal management, EV charging and driving in cold weather. An offline problem is also solved considering both electrical and thermal behaviors of EV battery\cite{apec_sarofim}. While effective at the individual EV level, both methods fail to consider the diverse charging behaviors (e.g, EV arrival times) of a large fleet of EVs, which is a critical source of uncertainty that cannot be ignored. Moreover, both approaches rely on perfect knowledge of future information such as ambient temperature.
Reference \cite{ruan_temperature} develops a temperature-controlled charging scheme by considering the coupling between charging and heating. Though this approach captures the behavioral heterogeneity of EVs, it still performs offline decision-making and relies on modeling system uncertainties, which typically assumes sufficient data availability.

Different from \textit{offline} charging scheduling, \textit{online} charging control does not require complete future information such as EV behaviors and temperature conditions. Existing online charging algorithms only focus on the scenarios under perfect temperature conditions.
Examples include model predictive control (MPC) \cite{lee2021adaptive}. However, MPC still relies on predictive information, making its performance sensitive to the reliability of the forecasts. Rule-based methods (e.g., least-laxity-first and earliest-deadline-first \cite{subramanian2013real}) are computationally efficient and easy to implement. However, these algorithms rely on predefined heuristics, making them less robust to system uncertainties. Although reinforcement learning achieves online charging control without relying on the modeling of system uncertainties, it requires sufficient training to achieve satisfactory performance and cannot be easily extended to diverse scenarios~\cite{11694268}. Lyapunov optimization, which is based on queue theory, serves as a promising solution to online control, as it grants robustness to system uncertainties without any prior knowledge of them. Lyapunov optimization is an approach for maintaining queue stability while optimizing time-average performance under uncertain system dynamics. It turns time-average optimization problems into tractable per-slot deterministic problems using drift-plus-penalty techniques \cite{neely2010stochastic}.
Lyapunov optimization has been employed in various applications such as building temperature control \cite{zheng2014distributed} and online EV charging control \cite{jin2014optimized,zhou2018incentive}.

To summarize, in an online setting, how to design a coordinated EV battery charging and heating control algorithm that is robust against system uncertainties and independent of prior knowledge of such uncertainties is an open problem. We aim to fill the above research gap by extending the Lyapunov optimization from conventional online EV charging control to coordinated charging and heating. However, the following challenges arise:
\begin{itemize}
    \item Beyond the uncertainties in conventional EV charging problems, the design of coordinated charging and heating is further burdened by additional temperature-related uncertainties, complicating the overall system regulation.
    \item Due to the coupling between charging and heating, the battery heating timing and power requirements vary across different cold environments. Thus, it is difficult to adaptively perform EV battery temperature control to facilitate charging process for cost savings across diverse cold conditions.
    \item Under uncertain system dynamics, establishing a systematic framework to coordinately allocate heating and charging power based on the battery's dynamics is a highly non-trivial task.
\end{itemize}

Built upon the Lyapunov optimization framework, this work aims to tackle the above challenges by developing an efficient, robust, and system-level online coordinated charging and heating control approach for a charging station.
Our contributions are summarized as follows:

    \textbf{(C1)} We depict both EV battery charging and thermal dynamics by queue models. Notably, we propose a queue model for EV battery temperature control, which resembles an energy storage system.
    Unlike the existing queue model for temperature control that are restricted to on/off heating, our model extends the temperature control to continuous inputs, allowing for a more fine-grained regulation of temperature. Moreover, 
    this model enables temperature regulation through Lyapunov optimization based storage dynamics control, allowing us to integrate temperature control into the Lyapunov optimization framework for conventional  EV charging control. 
    
    \textbf{(C2)} Supported by the framework of Lyapunov optimization, we systematically achieve the coordination of battery charging and heating. To be specific, we design a novel Lyapunov function to simultaneously characterize the charging and heating dynamics by calculating the weighted sum of EV charging and heating queue backlogs. Then, we propose a drift-plus-penalty term for coordinately controlling the charging and heating power allocation while reducing the system energy cost.
    
    \textbf{(C3)} Unlike the existing coordinated charging and heating scheme that is implemented in an offline setting with future predictive information, we propose a novel online coordinated EV charging and heating control policy under linear programming for cost savings by minimizing the drift-plus-penalty term. This policy is robust with respect to cold temperature conditions, easy to implement, and prediction-free. We also prove that under our proposed policy, the feasibility of the EV battery temperature constraints can be guaranteed. Numerical results based on real-world data validate the advantages of the proposed coordinated charging and heating control framework over existing uncoordinated online methods.

The rest of this paper is organized as follows. Section \ref{sec:formulation} introduces the system framework and mathematical model. In section \ref{sec:online algo}, we construct queue systems and design our online control algorithm by Lyapunov optimization.
The numerical simulations are presented in Section \ref{sec:case}, and Section \ref{sec:conclusion} concludes this paper.

\section{Problem Formulation}
\label{sec:formulation}
In this section, we first illustrate the charging station system, and then model both charging and thermal dynamics of the system. Finally, we formulate the problem in an offline setting.
\subsection{System Framework}
\begin{figure}[!htpb]
    \centering
    \includegraphics[width = 0.95\linewidth]{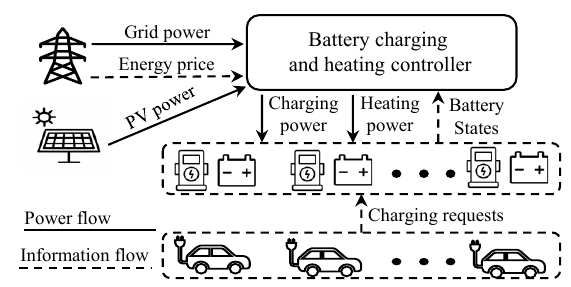} 
    \vspace{-1em}
    \caption{System model.}
    \label{fig:system_model}
\end{figure}
We consider a charging station system (see Fig.~\ref{fig:system_model}) in cold environments, where a station operator can simultaneously control both the charging and heating power of all available EVs in real time via bidirectional communication. Within this paradigm, the station operator aims to coordinate both EV battery charging and thermal dynamics, which are coupled in nature.
The charging station operates with local photovoltaic (PV) generation. Each time, when the self-generated PV power is insufficient to meet the power consumption of the entire system, the operator needs to purchase additional power from the grid, which incurs a cost. The goal of the station operator is to design effective coordinated EV battery charging and heating control policies to satisfy charging demands in various cold environments with minimum system costs. 

One would anticipate that the determination of such a good policy is complicated by various uncertainties under coupled system dynamics. There are multiple sources of uncertainty: EV charging sessions (including EV arrival times, charging deadlines, charging demands, and initial values of battery temperature), local PV generation, the evolution of ambient temperature, and the price the system operator pays for the energy from the grid. Here, we do not assume any prior knowledge about the stochastic behaviors of these uncertainties. Thus, each time, the system operator delivers a decision based only on the system information observed up to the current time. Under the above challenging setting, we are interested in designing a good policy to provide EV battery charging and heating services in a coordinated manner.

\subsection{System Modeling}
In this subsection, we introduce mathematical notations to describe the problem. Let $\mathcal{T}$ denote the set of discrete time slots and $\mathcal{I}$ as the set of EVs, where $t \in \mathcal{T}$ and $i \in \mathcal{I}$ index time and individual EVs, respectively. Denote by $p_{i,t}^c$ the charging power of EV $i$ at time $t$, which is restricted by its peak charging rate $\bar{p}_{i,t}^c$.
\begin{equation}
  0  \leq {p}_{i,t}^c \leq \Bar{p}_{i,t}^c, \forall{i},\forall t \in [t_i^a, t_i^d).
\end{equation}
We note that $\bar{p}_{i,t}^c$ is a time-varying parameter, which is affected by temperature conditions of EV $i$.

EV $i$'s charging session is initiated when EV $i$ arrives at the charging station, with the corresponding arrival time and initial battery energy denoted by $t_i^a$ and $E_i^{ini}$, respectively. EV $i$ specifies a deadline $t_i^d$ for its charging session with a desired battery energy $E_i^{dep}$ upon departure. It is assumed that all EVs leave at their deadlines, even if their desired state of charge (SoC) has not been met. The availability of EV charging sessions are defined as follows:
\begin{equation}
    p_{i,t}^c= 0, \forall{i},\forall{t} \notin [t_i^a,t_i^d).
\end{equation}
Let $E_{i,t}$ be the battery energy of EV $i$ at time $t$ such that
\begin{equation}
    E_{i,t_i^a} = E_i^{ini}, \forall{i}.
\end{equation}
The battery energy shall satisfy the following:
\begin{equation}\label{eq:energy_range}
    E_i^{dep} \leq E_{i,t_i^d} \leq \Bar{E}_{i}, \forall{i},
\end{equation}
where $\bar{E}_i$ is the battery energy limit of EV $i$.
The following transition function captures the system charging dynamics:
\begin{equation}
    E_{i,t+1} = E_{i,t} + \delta_c{p}_{i,t}^c\Delta t, \forall{i},\forall t \ne T
\end{equation}
where $\delta_c$ is the charging efficiency and $\Delta t$ is the time interval.

Denote by $p_{i,t}^h$ the heating power of EV $i$ at time $t$, which serves as a signal to control the heater inside EV $i$ for battery thermal management. $p_{i,t}^h$ is restricted by EV $i$'s peak heating rate $\bar{p}_{i,t}^h$, which is also time-varying and affected by temperature conditions.
\begin{equation}
    0 \leq {p}_{i,t}^h \leq \Bar{p}_{i,t}^h, \forall{i},\forall t \in [t_i^a, t_i^d).
\end{equation}
Similar to charging power control, we have
\begin{equation}
    p_{i,t}^h = 0, \forall{i},\forall{t} \notin [t_i^a,t_i^d).
\end{equation}
Let $T_{i,t}$ be the battery temperature of EV $i$ at time $t$ such that
\begin{equation}
     T_{i,t_i^{a}} = T_{i}^{ini}, \forall i,
\end{equation}
where $T_i^{ini}$ is EV $i$'s initial temperature state.
We use the transition function in \cite{ruan_temperature} to describe the EV battery heating dynamics as follows
\begin{equation}\label{eq:thermal_dynamics}
    q_i(T_{i,t+1} - T_{i,t}) = -\eta_i(T_{i,t} - T_{t}^0) + \delta_h p_{i,t}^h + (1 - \delta_c)p_{i,t}^c,
\end{equation}
where $q_i$ is EV $i$'s parameter determined by its battery mass, heat capacity and time interval; $\eta_i$ is EV $i$'s parameter determined by its thermal insulation coefficient, transfer coefficient and heat dissipation area of the battery; $\delta_h$ is the heating efficiency. Let $T_t^0$ be the ambient temperature at time $t$, and $-\eta_i(T_{i,t} - T_{t}^0)$ is associated to the heat dissipation to the external environment. $(1-\delta_c)p_{i,t}^c$ is associated to the heat generated during the charging process.
$\delta_hp_{i,t}^h$ is associated to the heating power transferred to the battery.

In cold weather, the temperature of EV $i$ shall be controlled within a preferred range to meet EV $i$'s charging demand
\begin{equation}\label{eq:temp_range}
 T_{i}^l \leq T_{i,t} \leq T_{i}^u, \forall i,t,
\end{equation}
where $T_i^l$ and $T_i^u$ are the lower and upper bounds the temperature range, respectively.

We utilize the model in~\cite{ruan_temperature} to characterize the impact of temperature on EV $i$'s peak charging and heating rate:
\begin{equation}\label{eq:temp_charging_limit}
    \Bar{p}_{i,t}^c = \Bar{p}_{i}^c + \beta_i^c T_{i,t},
\end{equation}
\begin{equation}\label{eq:temp_heating_limit}
    \Bar{p}_{i,t}^h =  \Bar{p}_{i}^h - \beta_i^h T_{i,t},
\end{equation}
where $\bar{p}_i^c$, $\bar{p}_i^h$, $\beta_i^c$ and $\beta_i^c$ are the related coefficients in above linear models. Equations \eqref{eq:temp_charging_limit} and \eqref{eq:temp_heating_limit} align with our empirical observations, indicating that lower temperatures reduce the charging rate and increase the heating rate, as captured by the temperature-dependent peak charging and heating rates.

The total power of charging and heating of EV $i$ is upper bounded by $\Bar{p}_{i}$.
\begin{equation}
p_{i,t}^c + p_{i,t}^h \leq \Bar{p}_{i}, \forall i,t.
\end{equation}

In the system level, the charging station can draw power from local PV generation and power grid. Let $p_{t}^{pv}$ be the power drawn from local PV generation. Let $p_t^g$ be the power drawn from the grid. Then, we have
\begin{equation}
     0 \leq p_{t}^{pv} \leq \Bar{p}_{t}^{pv}, p_t^g\geq 0 \; \forall t,
\end{equation}
where $\Bar{p}_{t}^{pv}$ is the upper bound of PV generation.

 The power balance of the entire system is described by:
\begin{equation}
     p_{t}^{pv}  + p_{t}^{g} =  \sum\nolimits_{i \in \mathcal{I}} ({p}_{i,t}^c + {p}_{i,t}^h), \; \forall t.
\end{equation}

 $C_t$ is the system cost of the charging station at time $t$:
\begin{equation}
    C_t = \lambda_tp_{t}^g \Delta t, \; \forall t,
\end{equation}
where $\lambda_t$ is energy price at time $t$. Since the charging station operator purchases energy from wholesale electricity market, the electricity price $\lambda_t$  is time-varying.

It should be noted that the complete model of the system has been presented without detailing the decision-making process. Before designing our desired control policies, we first approach the problem in an offline setting with the assumption of perfect future information, which leads to the following offline formulation:
\begin{equation}
    \begin{aligned}
  &\mathbf{P1\colon} \min \quad   \sum_{t = 1}^T C_t + \alpha\sum_{i \in \mathcal{I}} \left(E_{i,t_i^d} - E_i^{dep}\right)^2,\\
  &\textrm{s.t.} \quad  {(1) - (3), (5) - (16),E_{i,t} \le \bar{E}_i, \forall i,t.}
    \end{aligned}
\end{equation}
In $\mathbf{P1}$, term $\alpha\sum_{i \in \mathcal{I}} \left(E_{i,t_i^d} - E_i^{dep}\right)^2$ is a penalty for violating the constraint $E_i^{dep} \leq E_{i,t}$, with $\alpha>0$. The above penalized formulation is reasonable, because it is hard to ensure strict fulfillment of all EV charging demands, especially in cold environments. Note that this penalty term will not be effective when the battery energy exceeds $E_i^{dep}$, because one can always reduce the battery energy down to $E_i^{dep}$, which would lower the energy purchase cost while making the penalty term zero.  Here, we face multiple uncertainties, including future EV charging sessions ($t_i^a$, $t_i^d$, $E_i^{ini}$, $E_i^{dep}$), their initial battery temperatures $T_i^{ini}$, local PV generation limit $\bar{p}_t^{pv}$, ambient temperature $T_t^0$ and energy prices $\lambda_t$. In practice, it is difficult to obtain accurate prediction of these uncertainties. 
Thus, in next section, we are going to develop an efficient online policy to  solve $\mathbf{P1}$ without any prior knowledge of these uncertainties.

\section{Online Algorithm Design}
\label{sec:online algo}
In this section, we first depict the system dynamics by queue models. Then, we resort to the framework of Lyapunov optimization to design our online policy.

\subsection{Queue Modeling of Thermal Dynamics}
Lyapunov optimization is a mathematical framework that is closely connected to queueing theory. We observe that the evolution of temperature specified in \eqref{eq:thermal_dynamics} is similar to a queue system. Thus, we remodel the thermal dynamics by a queue model. It is reasonable to assume that in cold environments, the lower bound of the temperature range controlled for the EV battery is higher than the ambient temperature (i.e., $T_i^l > T_t^0$), as this helps accelerate the charging rate and better satisfy the charging demands. Let $\Delta T_{i,t}^d$ be the temperature loss of each time. Then, given the ambient temperature $T_t^0$, the battery temperature will decrease from $T_{i}^u$ to $T_{i}^l$ in $K_{i,t}$ timeslots when $p_{i,t}^h = 0$, and $p_{i,t}^c = 0$:
\begin{equation}\label{eq:theral_queue_1}
    T_{i}^l = T_{i}^u - K_{i,t}\Delta T_{i,t}^d, \forall i,t.
\end{equation}
By \eqref{eq:thermal_dynamics} and $p_{i,t}^h=p_{i,t}^c=0$, we have
\begin{equation}\label{eq:theral_queue_2}
    T_{i}^l = T_t^0(1 - \zeta_i^{K_{i,t}}) + T_{i}^u\zeta_i^{K_{i,t}},
\end{equation}
where $\zeta_i = 1 - \eta_i/q_i$.

By \eqref{eq:theral_queue_1} and \eqref{eq:theral_queue_2} , we can further derive
\begin{equation}\label{eq:theral_queue_3}
    K_{i,t} = \frac{1}{\log\zeta_i}\log\frac{T_t^0 - T_{i}^l}{T_{t}^0 - T_{i}^u},
\end{equation}
and
\begin{equation}\label{eq:theral_queue_4}
    \Delta T_{i,t}^d = \frac{(T_{i}^u - T_{t}^0)(1 - \zeta_i^{K_{i,t}})}{K_{i,t}}.
\end{equation}

Combing \eqref{eq:thermal_dynamics} and \eqref{eq:theral_queue_1}-\eqref{eq:theral_queue_4}, we obtain the queue model of the battery thermal dynamics as follows:
\begin{equation}\label{eq:themal_dyna}
    {T}_{i,t+1} = {T}_{i,t} - \Delta T_{i,t}^d +  \Delta T_{i,t}^c,
\end{equation}
where $\Delta T_{i,t}^c =  (\delta_h p_{i,t}^h + (1 - \delta_c)p_{i,t}^c)/q_i$, and $\Delta T_{i,t}^d$ is time-varying. It should be noted that the above thermal dynamics is approximated exclusively for the purpose of the subsequent Lyapunov optimization based algorithm design. During implementation, the temperature state evolution of EVs is strictly governed by the original, unapproximated dynamics~\eqref{eq:thermal_dynamics}.

\subsection{Queue System Construction}
\subsubsection{Deadline-Aware EV Charging Demand Queue}

We adopt the deadline-aware EV charging queue model as proposed in our previous work\cite{EV_queue}, which has been demonstrated to exhibit good and robust performance under uncertain system dynamics. The core idea of such a model is to group the EV charging demands based on EV owners' declared deadlines, so that at each time $t$, the EV charging demands with an equal remaining parking time or deadline would be allocated to the same queue. Let $\mathcal{I}_{t}^r$ be the set of EVs that share the same remaining parking time $r$ at time $t$. We note that  $r \in\{t_i^d - t,\forall t,i \}$ takes values from a discrete and finite set $\{1,\ldots,R\}$, as the remaining parking time of all EVs shall be upper bounded (i.e., $r \leq R$). Let $Q_t^r$ be the total charging demands of EVs in set $\mathcal{I}_t^r$ such that

\begin{equation}\label{eq:ev_queue_1}
    {Q_{t}^R = \sum\nolimits_{i \in \mathcal{I}_{t}^R} (E_i^{dep} - E_{i}^{ini})}.
\end{equation}
Let $\mathcal{I}_{arr,t}^r$ be the set of new EV arrivals observed during time slot $t$ with a remaining park time $r$. EVs in set $\mathcal{I}_{arr,t}^r$ are first available at time $t+1$ such that $\mathcal{I}_{arr,t}^r \subseteq \mathcal{I}_{t+1}^r$.
Then, we remodel the charging dynamics as follows:
\begin{equation}\label{eq:ev_queue_2}
    {Q_{t+1}^{r-1} =  \max \{Q_{t}^r - x_{t}^r,0\} + a_{t}^{r-1}, \forall ~2 \leq r \leq R},
\end{equation}
where
\begin{equation}\label{eq:ev_queue_a}
    {a_t^{r} = \sum\nolimits_{i \in \mathcal{I}_{arr,t}^{r}} (E_i^{dep} - E_{i}^{ini}), \forall~ 1 \leq r \leq R-1,}
\end{equation}
and
\begin{equation}\label{eq:ev_queue_x}
    x_t^r = \sum\nolimits_{i \in \mathcal{I}_{t}^r}\delta_cp_{i,t}^c \Delta t, \forall~ 1 \leq r \leq R.
\end{equation}
In  \eqref{eq:ev_queue_2}, $x_t^r$ represents the energy supplied to queue $Q_{t+1}^{r-1}$ at time $t$, which is determined by the policy we are going to design. $a_t^{r-1}$ is the new charging demands that arrive after we determine $x_t^r$ (but before we determine $x_{t+1}^{r-1}$). Hence, the information represented by $a_t^{r-1}$ is unknown when making decisions at time $t$, which means it is outside of our policy's control, despite its involvement in updating the system state. Moreover, our policy should clear the EV charging demands whenever possible, so the condition $\lim_{T \to \infty} {\mathbb{E}[Q_{T}^r]}/{T} = 0$ is imposed on EV charging demand queue $Q_t^r$ to achieve this goal, which means queue $Q_t^r$ is mean rate stable.

\subsubsection{Debt Queue}
We also construct a debt queue $Y_t$ to record the unfulfilled EV charging demands with $Y_0 = 0$, whose dynamics can be described as follows:
\begin{equation}\label{eq:debt_queue}
    Y_{t+1} = Y_{t} + Q_{t}^1 - x_{t}^1.
\end{equation}
By \eqref{eq:debt_queue}, we derive
\begin{equation}\label{eq:debt_queue_stability}
     \lim_{T \to \infty} \frac{\mathbb{E}[Y_{T}]}{T} =  \lim_{T \to \infty} \frac{\sum_{t=1}^{T-1}\mathbb{E}[Q_t^1 - x_t^1]}{T}.
\end{equation}
Equation \eqref{eq:debt_queue_stability} reveals that by controlling queue $Y_t$ to be mean rate stable (i.e., $\lim_{T \to \infty} {\mathbb{E}[Y_{T}]}/{T} = 0$), we can 
ensure that the charging decisions output by our policy complete the charging demands before the deadlines as much as possible.

\subsubsection{Virtual Temperature Queue} In this subsection, we construct a virtual temperature queue, and then relax the original temperature constraint  \eqref{eq:temp_range} into the queue stability condition of this virtual temperature queue.
For EV battery temperature control, we first relax constraint \eqref{eq:temp_range} by bounding the average battery temperature:
\begin{equation}\label{eq:ave_temp}
    T_i^l \leq \overline{T_{i,t}} \leq T_i^u,
\end{equation}
where $\overline{T_{i,t}}$ is the average EV battery temperature over time.

We construct a virtual temperature queue for each EV $i$
\begin{equation}\label{eq:temp_queue}
    H_{i,t} = T_{i,t} - \theta_i, \forall i,
\end{equation}
where $\theta_i$ is a perturbation parameter to be determined.

Obviously, queue $H_{i,t}$ evolves with dynamics as follows:
\begin{equation}\label{eq:temp_queue_dynamics}
    H_{i,t+1} = H_{i,t} - \Delta T_{i,t}^d +  \Delta T_{i,t}^c, \forall i
\end{equation}

By \eqref{eq:ave_temp} and \eqref{eq:temp_queue}, we have
\begin{equation}
    T_i^l - \theta_i \leq \overline{H_{i,t}}\leq T_i^u - \theta_i,
\end{equation}
which further implies the mean rate stability of the virtual temperature queue, i.e., $\lim_{T \to \infty} {\mathbb{E}[H_{i,T}]}/{T} = 0$.

Armed with the constructed queue models, we relax $\mathbf{P1}$ to $\mathbf{P2}$ as follows:
\begin{equation}
    \begin{aligned}
 & \mathbf{P2\colon}  \min \;  \lim_{T \to \infty} \frac{1}{T} \sum_{t = 1}^T \mathbb{E}[C_t],\\
 \quad \textrm{s.t.} \; & (1)-(3), (5)-(9), (11) - (16),\\
 &(23) - (27), H_{i,t_i^a} = T_{i,t_i^a} - \theta_i, \forall i,(31),\\
 & \text{Queues}\;Q_{t}^r\;\text{are mean rate stable}, \forall r,\\
 & \text{Queue}\;Y_t\;\text{is mean rate stable},\\
 & \text{Queues}\;H_{i,t}\;\text{are mean rate stable}, \forall i.
    \end{aligned}
\end{equation}
In $\mathbf{P2}$, we achieve approximate satisfaction of the original EV battery energy and temperature constraints \eqref{eq:energy_range} and \eqref{eq:temp_range} by maintaining queue stability. One intriguing observation of $\mathbf{P2}$ is that it allows us to coordinate battery charging and heating from the perspective of maintaining queue stability.

\subsection{Lyapunov Optimization}
Since Lyapunov optimization can ensure the queue stability constraints in $\mathbf{P2}$\cite{neely2010stochastic}, we design our control policy within the Lyapunov optimization framework.
Let $\mathbf{{\Theta}}_t$ be a vector that collects all queue backlogs. Then, we define a Lyapunov function as follows
\begin{equation}\label{eq:lya_function}
    L(\mathbf{{\Theta}}_t) \stackrel{\triangle}{=} \frac{1}{2} \left ( \sum_{r = 1}^{R} \frac{\gamma{Q_{t}^r}^2}{r+1}  + \gamma{Y_{t}}^2 + \sum_{i}H_{i,t}^2\right),
\end{equation}
where $\gamma$ is a given parameter. 
\eqref{eq:lya_function} serves as a scalar measure of $\mathbf{{\Theta}}_t$, characterizing simultaneously the charging and heating dynamics. We note that this Lyapunov function is deadline-differentiated, as the EV charging demand queues with less remaining time would be assigned higher weight parameters. This allows us to prioritize more urgent charging demands.

 Our decisions exert an effect on the value of $L(\mathbf{{\Theta}}_t)$. Thus, we use the Lyapunov drift to quantify the cost of a decision:
\begin{equation}\label{eq:drift}
    \Delta (\mathbf{{\Theta}}_t) \stackrel{\triangle}{=} \mathbb{E}[L(\mathbf{{\Theta}}_{t+1}) - L(\mathbf{{\Theta}}_t) ].
\end{equation}

By minimizing $\Delta (\mathbf{{\Theta}}_t)$, we are capable of steering the system away from heavily congested queue states, maintaining the stability of all queues \cite{neely2010stochastic}. This is exactly what we are pursuing. Actually, we minimize an upper bound of $\Delta (\mathbf{{\Theta}}_t)$ instead of the exact term, which makes the policy derivation more tractable. We proceed to find such an upper bound.

For EV charging demand queues, we have:
\begin{align*}
       & \quad \; \frac{1}{2} \left ( \sum_{r = 1}^{R} \frac{\gamma}{r+1} {Q_{t+1}^r}^2 - \sum_{r = 1}^{R}\frac{\gamma}{r+1} {Q_{t}^r}^2\right) \\
      & = \frac{1}{2}  \sum_{r = 1}^{R} \frac{\gamma}{r+1} \left ( ({\max\{Q_t^{r+1} - x_t^{r+1},0\} + a_t^{r}})^2 - {Q_{t}^r}^2 \right)  \\
      & \leq \frac{1}{2}  \sum_{r = 1}^{R} \frac{\gamma}{r+1} \left({Q_{t}^{r+1}}^2 - {Q_{t}^r}^2 +  {x_{t}^{r+1}}^2 + {a_t^r}^2 \right) \stepcounter{equation}\tag{\theequation}\label{demand_queue}\\
       & + \frac{1}{2}   \sum_{r = 1}^{R} \frac{\gamma}{r+1} \left(2 Q_{t}^{r+1}A_t^r - 2 Q_{t}^{r+1}x_{t}^{r+1} \right) .
\end{align*}
The derivation above stems from the observation that for any $q \geq 0, b \geq 0, a \geq 0$, we have the following:
\begin{equation} \label{eq:inequality_bound}
    (\max[p - b,0]+a)^2 \leq p^2 + a^2 + b^2 + 2p(a - b).
\end{equation}
For the debt queue, we have:
\begin{align*}
       & \quad \; \frac{1}{2} \gamma({Y_{t+1}}^2 - {Y_{t}}^2)\\
       &  = \frac{1}{2} \gamma \left( {(Y_{t}+Q_{t}^1 - x_{t}^1)}^2 - {Y_{t}}^2   \right) \stepcounter{equation}\tag{\theequation}\label{debt_queue} \\
       & = \frac{1}{2} \gamma \left ( ({Q_{t}^1})^2 + ({x_{t}^1})^2 - 2Q_{t}^1x_{t}^1 +2Y_{t}(Q_{t}^1 - x_{t}^1) \right).
\end{align*}
For virtual temperature queues, we have:
\begin{align*}
       & \quad \; \frac{1}{2} \sum\nolimits_i({H_{i,t+1}}^2 - {H_{i,t}}^2)\\
       &  = \frac{1}{2} \sum\nolimits_i\left( {(H_{i,t} - \Delta T_{i,t}^d + \Delta T_{i,t}^c)}^2 - {H_{i,t}}^2   \right) \stepcounter{equation}\tag{\theequation}\label{temperature_queue} \\
       & = \frac{1}{2} \sum\nolimits_i \left ( (\Delta T_{i,t}^c - \Delta T_{i,t}^d)^2 + 2H_{i,t}(\Delta T_{i,t}^c - \Delta T_{i,t}^d) \right).
\end{align*}

The cost of the energy purchase resulting from the decision should also be considered. To balance the energy cost and decision cost, we introduce the drift-plus-penalty function with the obtained upper bound of $\Delta (\mathbf{{\Theta}}_t)$:
    \begin{align*}
       & \quad \; \Delta (\mathbf{{\Theta}}_t) + VC_t \leq  VC_t \\
       & + \frac{1}{2} \left  ( \sum_{r = 1}^{R} \frac{\gamma}{r+1} ({Q_{t}^{r+1}}^2 - {Q_{t}^r}^2)  \right.\\
       & + \sum_{r = 1}^{R} \frac{\gamma}{r+1} ({a_t^r}^2 +2Q_t^{r+1}a_t^r) +\gamma {Q_{t}^1}^2 \stepcounter{equation} \tag{\theequation}\label{upper_bound}\\
       & \left. +  \sum_{r = 1}^{R} \frac{\gamma}{r} {x_{max}^{r}}^2 - 2\sum_{r = 1}^{R} \frac{\gamma}{r}Q_{t}^{r}x_{t}^{r} + 2\gamma Y_{t}(Q_{t}^1 - x_{t}^1)\right)\\
       & + \frac{1}{2} \sum\nolimits_i \max [(\Delta T_{i,max}^c)^2,(\Delta T_{i,max}^d)^2]\\
       & + \sum\nolimits_i \left ( H_{i,t}(\Delta T_{i,t}^c - \Delta T_{i,t}^d)\right).
    \end{align*}
where $V$ is a  given parameter to control the trade-off between the energy cost and queue stability. $x_{max}^r$, $\Delta T_{i,max}^c$, and $\Delta T_{i,max}^d$ are the upper bounds of $x_t^r$, $ \Delta T_{i,t}^c$, and $ \Delta T_{i,t}^d$.

\subsection{Online EV Charging Control}
We obtain our desired control policy by minimizing the term on the right side of \eqref{upper_bound}, with constant terms omitted.
\begin{align*}
     \mathbf{P3\colon}  {\min_{p_{i,t}^c,p_{i,t}^h,\forall i}}  \quad  & VC_t  - \sum_{r = 1}^{R} \left (\frac{\gamma}{r}Q_{t}^{r}x_{t}^{r} \right)- \gamma Y_{t} x_{t}^1 \\
    &  + \sum\nolimits_i \left ( H_{i,t}\Delta T_{i,t}^c\right) \stepcounter{equation}\tag{\theequation}\label{obj_func},
\end{align*}
 where $x_t^r = \sum\nolimits_{i \in \mathcal{I}_{t}^r}\delta_cp_{i,t}^c \Delta t$ and $\Delta T_{i,t}^c =  (\delta_h p_{i,t}^h + (1 - \delta_c)p_{i,t}^c)/q_i$. The problem is subject to the following physical constraints:
\begin{subequations}
    \begin{align}
    & 0 \leq {p}_{i,t}^c \leq \Bar{p}_{i,t}^c, \forall i,t \label{eq:rt_cons_1} \\ 
    & 0 \leq {p}_{i,t}^h \leq \Bar{p}_{i,t}^h, \forall i,t \label{eq:rt_cons_2} \\
    & p_{i,t}^c + p_{i,t}^h \leq \Bar{p}_{i}, \forall i,t \label{eq:rt_cons_3}\\
&0 \leq p_{t}^{pv} \leq \Bar{p}_{t}^{pv},  p_t^g \geq 0,\forall t \label{eq:rt_cons_4} \\
&   p_{t}^{pv}  + p_{t}^{g} =  \sum\nolimits_{i} ({p}_{i,t}^c + {p}_{i,t}^h), \forall t \label{eq:rt_cons_5}
    \end{align}
\end{subequations}

The above minimization problem $\mathbf{P3}$ is a linear programming with decision variables $p^c_{i,t},p^h_{i,t},\forall i$, which can be easily tackled by the optimization solvers. At each time slot, the system operator updates $Q_t^r$, $Y_t$, and $H_{i,t}$ by \eqref{eq:ev_queue_2}, \eqref{eq:debt_queue}, and \eqref{eq:temp_queue_dynamics}, updates peak charging rate $\bar{p}_{i,t}^c$ and heating rate $\bar{p}_{i,t}^h$ by \eqref{eq:temp_charging_limit} and \eqref{eq:temp_heating_limit}, observes real-time energy price $\lambda_t$ and local PV generation limit $\bar{p}_t^{pv}$, and then solves  $\mathbf{P3}$ to coordinately make charging decision $p_{i,t}^{c \ast}$ and heating decision $p_{i,t}^{h \ast}$ for all available EVs. 
We would like to mention that our policy is implemented without any prior knowledge about the stochastic behaviors of the involved uncertainties.

We remark that different from offline methods, our policy does not require incorporating the two linear relationships \eqref{eq:temp_charging_limit} and \eqref{eq:temp_heating_limit} into the constraints of $\mathbf{P3}$, which implies that other nonlinear relationships can still work under our policy. Moreover, the dynamics of $H_{i,t}$ is governed by~\eqref{eq:temp_queue_dynamics}, whereas the true temperature state $T_{i,t}$ follows the original, unapproximated dynamics~\eqref{eq:thermal_dynamics}. Consequently, $H_{i,t}$ serves as a virtual representation of the temperature state $T_{i,t}$ that facilitates effective EV battery temperature control.


\section{Theoretical Guarantee}
In this section, we explore two theoretical properties of the proposed approach: feasibility and optimality.
\subsection{Feasibility Guarantee}
The decisions made by solving $\mathbf{P3} $ may not satisfy the original temperature constraint \eqref{eq:temp_range}. However, if the parameter $\theta_i$ is selected under the guidance of the ensuing theorem, then we can guarantee that the temperature constraint holds under our policy.
 We first make the following assumptions for each EV $i$ at each time slot $t$:
\begin{equation}\label{eq:assum1}
      (1-\delta_c)\bar{p}_{i,t}^{c}/q_i - \Delta T_{i,t}^d \leq 0;
\end{equation}
\begin{equation}\label{eq:assum2}
    (\delta_h \min[\bar{p}_{i,t}^h,\bar{p}_{i,t} - \bar{p}_{i,t}^{c}] )/q_i - \Delta T_{i,t}^d \geq 0.
\end{equation}

We give some explanations on the above assumptions. \eqref{eq:assum1} indicates that, when the heating power is zero, the heat generated during EV battery charging process is always less than the heat loss to the ambient environments so that the temperature will not increase. This assumption is reasonable in a cold climate. \eqref{eq:assum2} can be seen as that, when we heat the EV battery at the available maximum heating rate $\min[\bar{p}_{i,t}^h,\bar{p}_{i,t} - \bar{p}_{i,t}^{c}]$, the battery temperature will increase. Such a condition can be easily satisfied as we should always reserve some power capacity for heating in cold climates.
The following theorem characterizes the feasibility of the proposed method.
\begin{theorem} \label{theorem_1}
    Suppose that the energy prices $\lambda_t$ are upper bounded by $\bar{\lambda}$, and $T_i^l \leq T_{i,1} \leq \theta_i + \max_t\{ \Delta T_{i,t}^c - \Delta T_{i,t}^d\} \leq T_i^u$. If $\theta_i$ and V satisfy:

\begin{equation}
    \theta_i = q_iV\bar{\lambda}\Delta t/\delta_h + \max\nolimits_t\{ \Delta T_{i,t}^d - \Delta T_{i,t}^c\} + T_{i}^l
\end{equation}
\begin{equation}
     0 < V \leq V_{max}
\end{equation}
where

$ V_{max}= \min\nolimits_i\frac{T_i^u - T_i^l- \max\nolimits_t\{ \Delta T_{i,t}^c - \Delta T_{i,t}^d\} -\max\nolimits_t\{ \Delta T_{i,t}^d - \Delta T_{i,t}^c\}}{q_i\bar{\lambda}\Delta t/\delta_i^h}$
then the temperature constraint $T_{i}^l \leq T_{i,t} \leq T_i^u, \forall i,t$ holds under the proposed policy.
\end{theorem}

The proof of Theorem \ref{theorem_1} can be found in Appendix A. In the assumption of Theorem \ref{theorem_1}, $\bar{\lambda}$, $\max_t\{ \Delta T_{i,t}^c - \Delta T_{i,t}^d\}$, and $\max_t\{ \Delta T_{i,t}^d - \Delta T_{i,t}^c\}$ can be estimated by historical data and problem parameters. Moreover, $T_i^l \leq T_{i,1} \leq \theta_i + \max_t\{ \Delta T_{i,t}^c - \Delta T_{i,t}^d\} \leq T_i^u$ can be met by enlarging the interval between the thermal limits $T_i^u$ and $T_i^l$. Intuitively, this broadened temperature regulation range offers more flexibility than a fixed setpoint approach, enabling the charging system to better accommodate various ambient temperature conditions.

\subsection{Optimality Guarantee}
 We can also characterize the optimality of our policy.
\begin{theorem}\label{theorem_2}
Denote the expected time-average cost under the proposed policy and optimal policy by $c^\ast$ and $\hat{c}$, respectively.
Then we have
    \begin{equation}
        c^\ast \leq \hat{c} + \frac{B}{V},
    \end{equation}
    where B is a constant.

\end{theorem}

The proof of Theorem \ref{theorem_2} can be found in Appendix B. 
We see that the performance of the proposed policy is affected by the value of $V$. A larger value of $V$ leads to lower cost at the expense of weaker control over the queue stability, as $V$ adjusts the importance of the energy cost in the objective function of $\mathbf{P3}$.

\section{Numerical Experiments}
\label{sec:case}

\subsection{Simulation Setup}
This section presents a performance evaluation of the proposed online EV charging control policy in comparison with following representative schemes.
\begin{itemize}
    \item \textbf{Offline}: The offline scheme is obtained by solving $\mathbf{P1}$ with perfect future information of all uncertainties. Thus, it represents an ideal reference. The offline scheme enforces a temperature constraint between 0 $^\circ$C and 20 $^\circ$C.
    \item \textbf{Proposed}: The proposed scheme controls the charging and heating in a coordinated manner. It aims to maintain the temperature within a range of  0 $^\circ$C to 20 $^\circ$C.
    \item \textbf{B1}: The B1 controls the charging and heating separately. To be specific, the B1 employs the method in \cite{EV_queue} to perform EV charging control. For battery thermal management, it aims to regulate the set point of the EV battery temperature within a comfort range of 9.5 $^\circ$C to 10.5 $^\circ$C, by using the control scheme in \cite{temp_set_point}.  This heating control strategy is also called bang-bang control\cite{hazan2022introduction}.
    \item \textbf{B2}: The charging control in B2 allows all EVs to charge at their time-varying peak charging rate, and its heating control is the same as that of B1.
    \item \textbf{NoHeat}: This scheme applies the smart charging method in \cite{EV_queue} for charging control in the absence of EV battery temperature regulation.
\end{itemize}

We discretize the 24-hour horizon into 5-minute slots. We adopt the EV arrival and departure times of workplace charging sessions from the Caltech ACN dataset~\cite{ACN_dataset}, which includes 45 individual EVs.
 We use the real-world electricity prices~\cite{price_data}, and PV generation data~\cite{pv_data}. The initial SoC of EVs is randomly generated from [0.1, 0.3]. The departure SoC of all EVs is set to 0.9. The charging efficiency $\delta_c$ is 0.95. The initial temperatures of EVs are randomly generated. During implementation, if an EV's initial temperature falls below the lower temperature bound $T_i^l$, then a preheating strategy is employed to heat the battery at the maximum available heating power; once the temperature reaches $T_i^l$, temperature regulation is taken over by the proposed algorithm.
 For simplicity, we test our algorithm with identical EV parameters. However, our method can easily tackle different EV battery characteristics. The capacity of all EV batteries is set to 50 kWh. The parameters introduced due to battery thermal management are adopted directly from the settings in \cite{ruan_temperature}, i.e., $\delta_h = 0.8$, $\bar{p}_i^c = 4.8$, $\beta_i^c = 0.12$, $\bar{p}_i^h = 3.0$, and $\beta_i^h = 0.024$. The total power limit of heating and charging $\bar{p}_i$ is 7.4 kW. In our setting, $q_i$ is 0.72 and $\eta_i$ is 0.048. The parameter settings of our policy are $V$ = 600, and $\gamma = 20$. We use the Canadian ambient temperature data from~\cite{temp_data}.

\subsection{Performance Analysis}

\begin{figure}[!htpb]
    \centering
    \includegraphics[width = 0.8\linewidth]{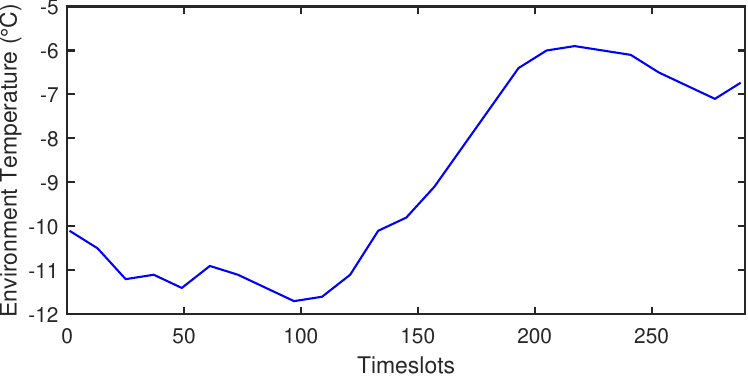} 
    \vspace{-1em}
    \caption{Ambient temperature profile on a specific day.}
    \label{fig:envir_tmp_profile}
\end{figure}

In this subsection, we conduct a case study to validate the benefits of our coordinated charging and heating control scheme. The ambient temperature data we use, is shown in Fig.~\ref{fig:envir_tmp_profile}. In TABLE~\ref{tab:method_comp}, we present four metrics to compare the performance of different methods, which are total cost, fulfillment ratio, cost index and heating ratio. The total cost refers to the total energy cost of 45 EVs for purchasing electricity from the grid; the fulfillment ratio is defined as the ratio of the total energy charged into all EV batteries to the total EV energy demands; the cost index, calculated as the total cost divided by the fulfillment ratio, reflects the system cost incurred per 1\% of the charging demands fulfilled; the heating ratio is defined as the proportion of the energy for heating all EV batteries in the total energy consumed for both charging and heating. 

\begin{figure*}[!htpb]
    \centering
    \includegraphics[width = 0.85\linewidth]{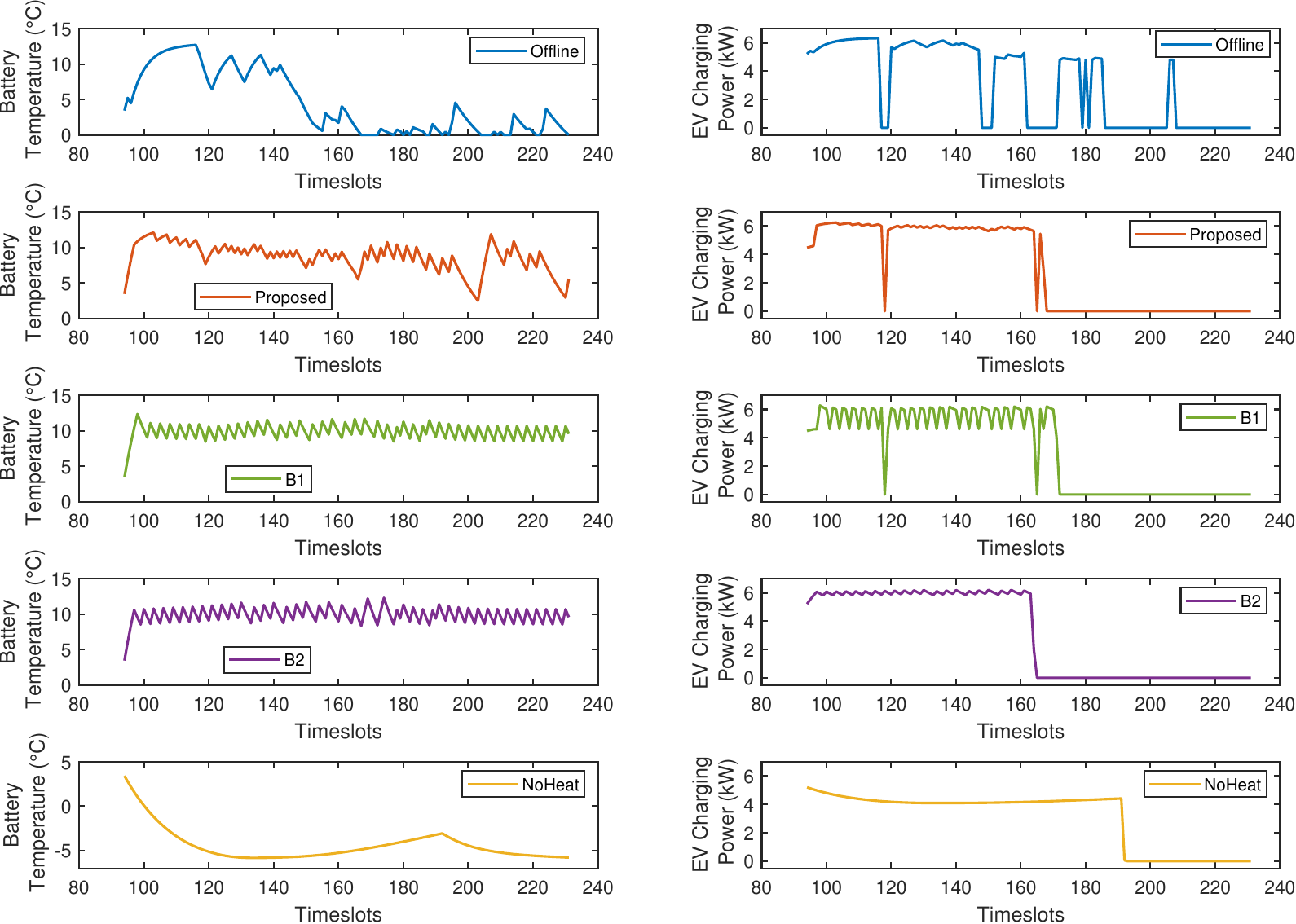} 
    \vspace{-1em}
    \caption{EV battery temperature and charging power profiles under different methods.}
    \label{fig:temp_charging_profile}
\end{figure*}

TABLE~\ref{tab:method_comp} shows that offline scheme achieves the best performance in terms of all metrics, as it has perfect knowledge of all future uncertainties. Our proposed method achieves a cost that is 12.2\% lower than B1 and 17.3\% lower than B2, demonstrating a significant advantage in terms of total cost. Such a cost saving arises from lower heating ratio and, more importantly, from the coordinated charging and heating control paradigm.
Since B2 does not respond to time-varying energy prices, it maintains a slight advantage in fulfillment ratio over ours. However, it has the highest cost index, showing poor economic efficiency in energy utilization. NoHeat achieves the lowest fulfillment ratio since the absence of battery thermal management leads to the slow charging rate. It is worth noting that despite consuming more energy for heating purposes, our approach fulfills more charging demands while incurring lower costs compared to the NoHeat. This intriguing observation can be attributed to the fact that in addition to enhancing the charging rate, coordinated charging and heating enables the charging station to strategically respond to price signals, resulting in lower total cost.

\begin{table}[!htpb]
\small
    \centering
    \caption{Results under different methods on a specific day}
    \vspace{-0.5em}
    \begin{tabular}{@{}ccccccccc@{}}
        \toprule
   \multirow[c]{2}{*}{Method} &  {Total} & {Fulfillment} & {Cost} & {Heating}  \\
       & Cost (\$) & Ratio (\%) & Index (\$/\%)& Ratio (\%)  \\
        \midrule
        Offline & 76.30  & 99.65 &0.7657 &12.63\\
        Proposed & 91.44  & 98.73 & 0.9262 &16.47\\
        B1 & 104.19  & 98.68 &1.0558 &18.15 \\
        B2 & 110.58  & 99.30 &1.1136 &18.01 \\
        NoHeat & 101.07  & 96.72 &1.0450 &0.00 \\
        \bottomrule
        \label{tab:method_comp}
   \end{tabular}
\end{table}

In Fig.~\ref{fig:temp_charging_profile}, we show the battery temperature and charging power profiles of one individual EV (randomly selected from 45 EVs) under different methods. The offline scheme exhibits irregular profiles in both temperature and charging power. Our method is able to maintain the temperature within the predefined range of of  0 $^\circ$C to 20 $^\circ$C, confirming the validity of Theorem~\ref{theorem_1}. 
Although approximated thermal dynamics is employed for the algorithm design, the actual temperature consistently satisfies the temperature constraints well, which justifies our proposed temperature queue model.
Compared with B1, our method presents a less fluctuating charging power profile due to coordinated charging and heating. B1 and B2 exhibit similar temperature profiles, as they employ the same heating control strategy. Without heating, the temperature of NoHeat stays at a low level, which in turn leads to the slow charging rate.

\subsection{Sensitivity Analysis}
This subsection conduct sensitivity analysis by exploring the impact of several factors on the performance of our method.

We change the value of parameter $V$ and parameter $\gamma$ to observe their effects, and the corresponding results are shown in Fig.~\ref{fig:V_gamma_impact}. We see that as $V$ increases, the total cost decreases. This can be understood from the fact that $V$ controls the importance of the energy cost in the objective function of $\mathbf{P3}$. This observation is also consistent with the conclusion in Theorem \ref{theorem_2}. Then we adjust the value of parameter $\gamma$ to evaluate its effect. It can be seen that an increase in parameter $\gamma$ leads to an increase in the fulfillment ratio.
Such a phenomenon arises from the influence of $\gamma$ in the Lyapunov function \eqref{eq:lya_function}, namely, to modulate the relative significance of charging queues within the decision cost. Therefore, when $\gamma$ becomes larger, our policy tends to perform stronger control over the charging queue stability, leading to better performance in terms of fulfillment ratio.

\begin{figure}[!htpb]
    \centering
    \includegraphics[width = 0.95\linewidth]{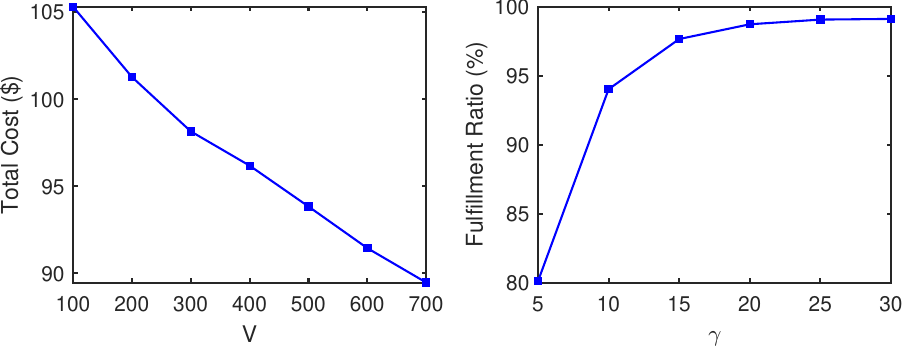} 
    \vspace{-1em}
    \caption{Impact of parameter V and parameter $\gamma$.}
    \label{fig:V_gamma_impact}
\end{figure}


To generate more scenarios, the entire temperature profile in Fig.~\ref{fig:envir_tmp_profile} is shifted by different offsets within a range of -12 $^\circ$C to 4 $^\circ$C, and each translated version is tested. In Fig.~\ref{fig:temp_shift_impact}, we report the results of all online schemes for comparison. When the environment becomes colder, the proposed policy, B1 and B2 can maintain a high level of fulfillment ratio owing to their battery thermal management strategies. 
One can observe the cost index and heat ratio of the proposed policy, B1 and B2 steadily increase as the ambient temperature decreases. Notably, our method consistently outperforms B1 and B2 in terms of cost index throughout the temperature variation, indicating that our method is both climate-resilient and economically efficient.
However, as the temperature decreases, the fulfillment ratio of NoHeat is obviously dropping down due to lack of battery heating. While achieving the lower cost index than B1 and B2 when the ambient temperature drops down, NoHeat sacrifices a considerable amount of charging demands.

To further show the robustness of our policy, we simulate three consecutive days in January, whose temperature profiles are shown in Fig.~\ref{fig:impact_envir_temp_profil}. We present the relevant results in TABLE~\ref{tab:method_comp_three} with comparison. Similar conclusions can be drawn as those observed in TABLE~\ref{tab:method_comp}. Specifically, our method shows strong climate resilience as well as superior energy-economic performance.

\begin{figure}[!htpb]
    \centering
    \includegraphics[width = 0.95\linewidth]{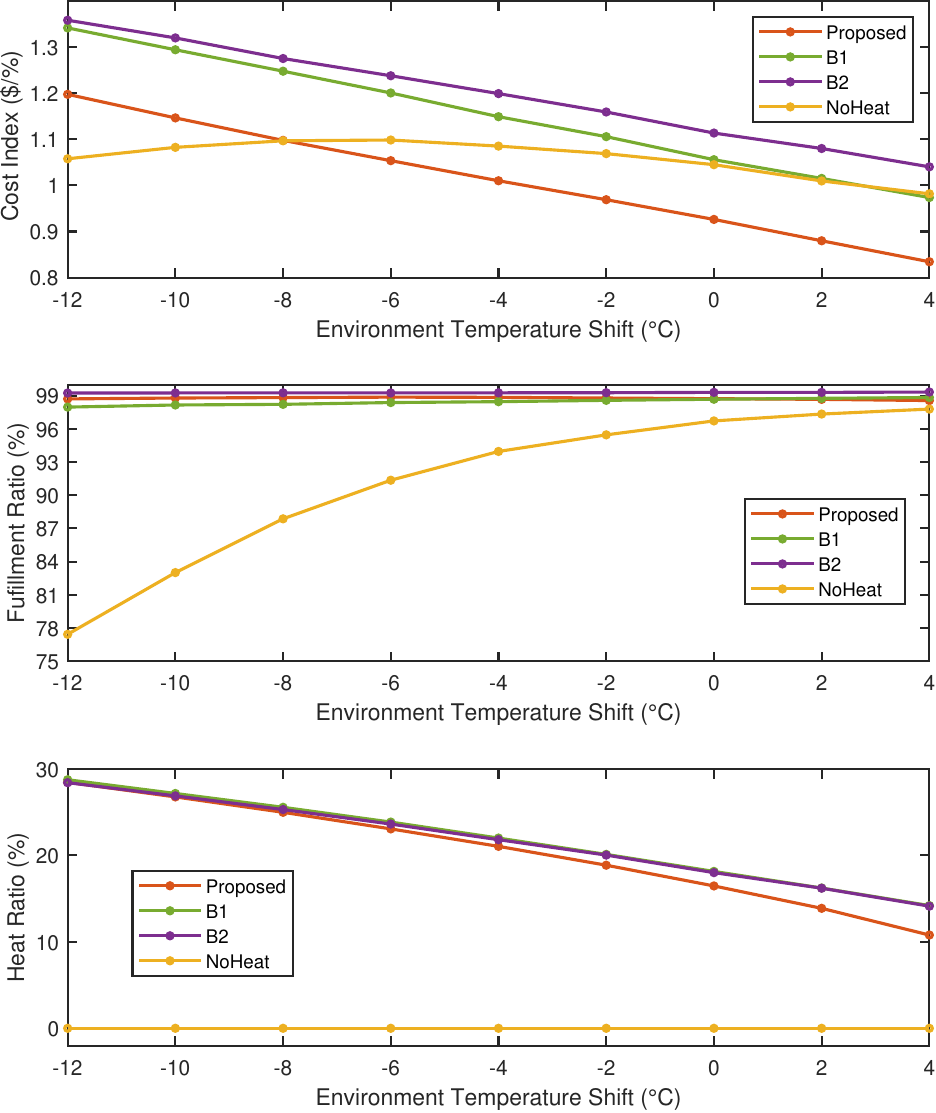} 
    \vspace{-1em}
    \caption{Impact of temperature on different online charging schemes.}
    \label{fig:temp_shift_impact}
\end{figure}

\begin{figure}[!htpb]
    \centering
    \includegraphics[width = 0.9\linewidth]{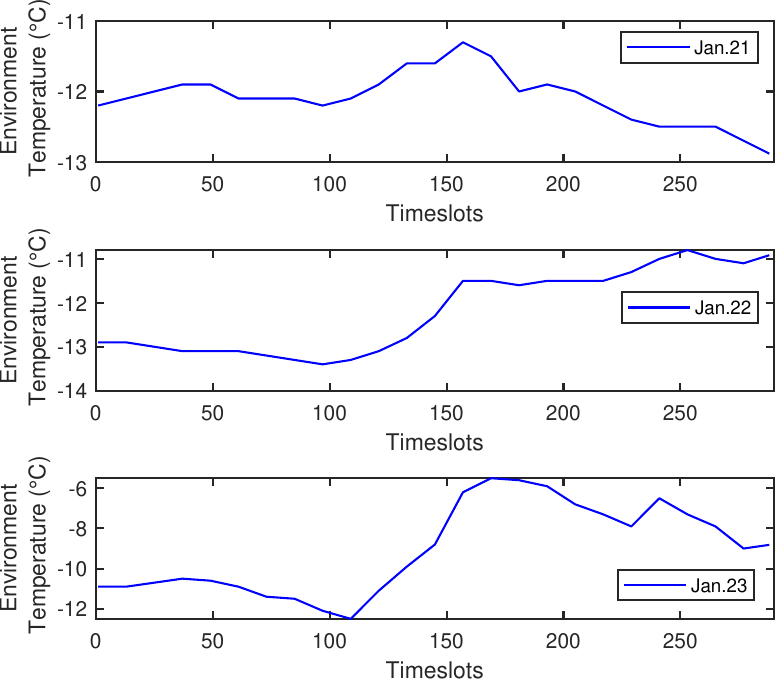} 
    \vspace{-1em}
    \caption{Ambient temperature profiles for three consecutive days.}
    \label{fig:impact_envir_temp_profil}
\end{figure}
\begin{table}[!htpb]
\small
    \centering
    \caption{Results under different methods for three consecutive days}
    \vspace{-0.5em}
    \begin{tabular}{@{}ccccccccc@{}}
        \toprule
  \multirow[c]{2}{*}{Date}& {Average} & \multirow[c]{2}{*}{Method} &  {Cost} & {Fulfillment }  \\
 &{Temperature}&  &Index (\$/\%) & Ratio (\%)\\
        \midrule
       \multirow[c]{5}{*}{Jan 21} & \multirow[c]{5}{*}{-12.0 $^\circ$C} & Offline & 0.8463  & 99.54\\
       & & Proposed & 1.0485  & 99.09 \\
       & & B1 & 1.1666  & 98.45  \\
       & & B2 & 1.2176  & 99.26  \\
       & & NoHeat & 1.0484  & 94.06 \\
       \midrule
       \multirow[c]{5}{*}{Jan 22} & \multirow[c]{5}{*}{-12.2 $^\circ$C} & Offline & 0.8403  & 99.54 \\
      & & Proposed & 1.0382  & 99.07 \\
      & & B1 & 1.1609  & 98.51  \\
      & & B2 & 1.2095  & 99.26  \\
      & & NoHeat & 1.0604  & 94.11 \\
       \midrule
       \multirow[c]{5}{*}{Jan 23} & \multirow[c]{5}{*}{-9.0 $^\circ$C}& Offline & 0.7642  & 99.65 \\
      & & Proposed & 0.9486  & 98.73 \\
      & & B1 & 1.0581  & 98.68  \\
      & & B2 & 1.117  & 99.30  \\
      & & NoHeat & 1.0184  & 96.72 \\
        \bottomrule
        \label{tab:method_comp_three}
   \end{tabular}
\end{table}

\subsection{Case Study Under an Extremely Cold Environment}
We validate the robustness of our policy under an extremely cold environment. We conduct a case study under the temperature profile depicted in Fig.~\ref{fig:extreme_envir_tmp_profile}, and the results are presented in TABLE~\ref{tab:method_comp_extreme}. In such an extremely cold environment, offline scheme, our method, B1 and B2 adaptively allocate more energy for battery thermal management, resulting in the heating ratio over 30\%. Compared with B1 and B2, our method still maintains an advantage in terms of cost index, although this advantage is less pronounced under the extremely cold condition. The fulfillment ratio of NoHeat is even below 50\%, indicating the necessity of battery thermal management when EVs are charged in extremely cold environments. NoHeat achieves the lowest cost index because it avoids additional energy consumption for heating, thereby requiring electricity purchases only during low-price periods. However, given the substantial unfulfilled charging tasks, this economic advantage becomes practically insignificant.

\begin{figure}[!htpb]
    \centering
    \includegraphics[width = 0.8\linewidth]{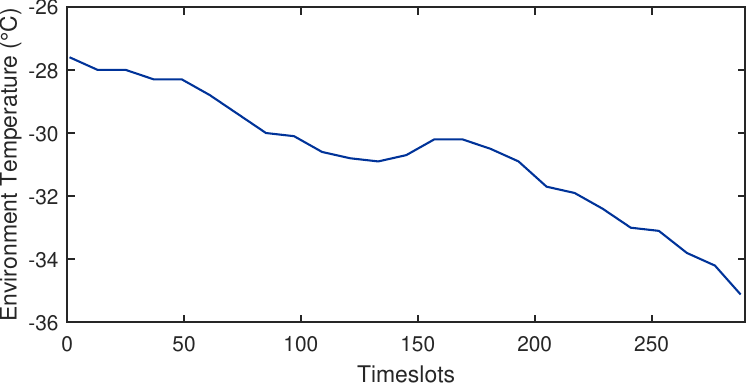} 
    \vspace{-1em}
    \caption{Ambient temperature profile under an extremely cold condition.}
    \label{fig:extreme_envir_tmp_profile}
\end{figure}

\begin{table}[!htpb]
\small
    \centering
    \caption{Results under different methods under an extremely cold condition}
    \vspace{-0.5em}
    \begin{tabular}{@{}ccccccccc@{}}
        \toprule
   \multirow[c]{2}{*}{Method} &  {Total} & {Fulfillment} & {Cost} & {Heating}  \\
       & Cost (\$) & Ratio (\%) & Index (\$/\%)& Ratio (\%)  \\
        \midrule
        Offline & 122.14  & 98.62 &1.2384 &31.70\\
        Proposed & 149.47  & 97.65 & 1.5308 &36.60\\
        B1 & 157.89  & 97.08 &1.6264 &36.52 \\
        B2 & 159.03  & 99.18 &1.6034 &36.00 \\
        NoHeat & 33.17  & 45.94 &0.7220 &0.00 \\
        \bottomrule
        \label{tab:method_comp_extreme}
   \end{tabular}
\end{table}

\section{Conclusion}
\label{sec:conclusion}
In this paper, we propose a novel algorithm to address the EV charging control problem in cold environments. The most distinctive feature of our algorithm is its ability to achieve coordinated control of charging and heating without any prior knowledge of uncertainties. Through a queue-based remodeling of the system dynamics, we achieve such coordinated control under the theoretical framework of Lyapunov optimization, where the queue stability conditions serve as the foundation for achieving joint control of charging and heating. Numerical results demonstrate the effectiveness of our algorithm by comparing it with baseline algorithms. 






\renewcommand\theequation{\Alph{section}.\arabic{equation}} 
\counterwithin*{equation}{section} 
\renewcommand\thefigure{\Alph{section}\arabic{figure}} 
\counterwithin*{figure}{section} 
\renewcommand\thetable{\Alph{section}\arabic{table}} 
\counterwithin*{table}{section} 

\bibliographystyle{IEEEtran}
\bibliography{references}


\begin{appendices}
\section{Proof of theorem \ref{theorem_1}}\label{appendix_A}
Let $p_{i,t}^{c\ast}$ and $p_{i,t}^{h\ast}$ be the optimal solution of $\mathbf{P3}$.

    We first prove the upper bound condition (i.e., $T_{i,t} \leq T_i^u$) using mathematical induction. We see that the upper bound condition holds at time 1 as we have $T_{i,1} \leq \theta_i + \max_t\{ \Delta T_{i,t}^c - \Delta T_{i,t}^d\} \leq T_i^u $. 
    Now we assume the upper bound condition holds at time $t$.
    \begin{enumerate}
        \item Suppose $T_{i,t} \leq \theta_i$. Since $\Delta T_{i,t}^c - \Delta T_{i,t}^d \leq \max_t\{\Delta T_{i,t}^c - \Delta T_{i,t}^d\}$, we obtain that $T_{i,t+1} \leq \theta_i + \max_t\{ \Delta T_{i,t}^c - \Delta T_{i,t}^d\} \leq T_i^u $.
        \item Suppose $T_{i,t} > \theta_i$. Then we have $H_{i,t} = T_{i,t} - \theta_i >0$, and thus we must have $p_{i,t}^{h\ast} = 0$. If $p_{i,t}^{h\ast} > 0$, then we can always decrease $p_{i,t}^{h\ast}$ to further decrease the objective value while still satisfying all the constraints by lowing the energy supply.
        Hence, $T_{i,t+1} = T_{i,t} + (1-\delta_c)p_{i,t}^{c\ast}/q_i - \Delta T_{i,t}^d \leq T_{i,t} +  (1-\delta_c)\bar{p}_{i,t}^{c}/q_i - \Delta T_{i,t}^d \leq T_{i,t} \leq \theta_i + \max_t\{ \Delta T_{i,t}^c - \Delta T_{i,t}^d\} \leq T_i^u$
    \end{enumerate}
    Then we should prove the lower bound condition $T_i^l \leq T_{i,t}$. Assume the lower bound condition holds at time $t$.
    \begin{enumerate}
        \item Suppose $T_{i,t} \geq T_i^l + \max_t\{\Delta T_{i,t}^d - \Delta T_{i,t}^c \}$. According to \eqref{eq:themal_dyna}, we have $T_{i,t+1} \geq T_{i,t} - \max_t\{\Delta T_{i,t}^d - \Delta T_{i,t}^c \} \geq T_i^l$.
        \item Suppose $T_{i,t} < T_i^l + \max_t\{\Delta T_{i,t}^d - \Delta T_{i,t}^c \}$. From the objective function of $\mathbf{P3}$, we obtain the following:
        \begin{equation}
            V\lambda_t\Delta t + \delta_h(T_{i,t} - \theta_i)/q_i < V\lambda_t\Delta t - V\bar{\lambda}\Delta t \leq 0
        \end{equation}
        In this case, the objective function of the linear programming $\mathbf{P3}$ is strictly decreasing with respect to variable $p_{i,t}^h$, and thus we must have $p_{i,t}^{h\ast} = \min[{p}_{i,t}^h,\bar{p}_{i,t} -{p}_{i,t}^{c\ast} ]$ from the constraints \eqref{eq:rt_cons_2} and \eqref{eq:rt_cons_3} . Hence, we have $T_{i,t+1} = T_{i,t} + (\delta_h \min[\bar{p}_{i,t}^h,\bar{p}_{i,t} - p_{i,t}^{c\ast}] + (1 - \delta_c)p_{i,t}^{c\ast})/q_i - \Delta T_{i,t}^d \geq T_{i,t} + (\delta_h \min[\bar{p}_{i,t}^h,\bar{p}_{i,t} - \bar{p}_{i,t}^{c}])/q_i - \Delta T_{i,t}^d \geq T_{i,t}$, which means the temperature will not decrease whenever $T_{i,t} < T_i^l + \max_t\{\Delta T_{i,t}^d - \Delta T_{i,t}^c \}$. Therefore, we have $T_{i,t+1} \geq T_{i,t} \geq T_i^l$.
    \end{enumerate}
    This completes the proof.
    
\section{Proof of theorem \ref{theorem_2}}\label{appendix_B}
     Denote  the corresponding charging and heating decisions of the optimal policy by  $\hat{p}_{i,t}^{c}$ and $\hat{p}_{i,t}^{h}$. 
Since $p_{i,t}^{c\ast}$ and $p_{i,t}^{h\ast}$ are the optimal solutions of $\mathbf{P3}$, we have the following inequality
       \begin{align*}
       & \quad \; \Delta (\mathbf{{\Theta}}_t^\ast) + VC_t^\ast \leq  V\hat{C}_t \\
       & + \frac{1}{2} \left  ( \sum_{r = 1}^{R} \frac{\gamma}{r+1} ({Q_{t}^{r+1}}^2 - {Q_{t}^r}^2)  \right.\\
       & + \sum_{r = 1}^{R} \frac{\gamma}{r+1} ({a_t^r}^2 +2Q_t^{r+1}a_t^r) +\gamma {Q_{t}^1}^2 \\
       & \left. +  \sum_{r = 1}^{R} \frac{\gamma}{r} {x_{max}^{r}}^2 - 2\sum_{r = 1}^{R} \frac{\gamma}{r}Q_{t}^{r}\hat{x}_{t}^{r} + 2\gamma Y_{t}(Q_{t}^1 - \hat{x}_{t}^1)\right)\\
       & + \frac{1}{2} \sum\nolimits_i \max [(\Delta T_{i,max}^c)^2,(\Delta T_{i,max}^d)^2]\\
       & + \sum\nolimits_i \left ( H_{i,t}(\Delta \hat{T}_{i,t}^c - \Delta T_{i,t}^d)\right)  \\
       & \leq V\hat{C}_t + \frac{\gamma}{2(R+1)}{Q_{max}^{R+1}}^2 + \frac{\gamma}{4}{Q_{max}^{1}}^2  \stepcounter{equation} \tag{\theequation}\\
       & + \sum_{r = 1}^{R} \frac{\gamma}{2(r+1)} ({a_{max}^r}^2 +2Q_{max}^{r+1}a_{max}^r) \\
       & +  \sum_{r = 1}^{R} \frac{\gamma}{2r} {x_{max}^{r}}^2 +\gamma Y_{max}Q_{max}^1 - \gamma Y_{t}\hat{x}_{t}^1 -\sum_{r = 1}^{R} \frac{\gamma}{r}Q_{t}^{r}\hat{x}_{t}^{r} \\
       & +  \frac{1}{2} \sum\nolimits_i \max [(\Delta T_{i,max}^c)^2,(\Delta T_{i,max}^d)^2]\\
       & +  \sum\nolimits_i \left ( H_{i,t}(\Delta \hat{T}_{i,t}^c - \Delta T_{i,t}^d)\right)  + \frac{\gamma}{2}\sum_{r=2}^R \frac{(Q_{\max}^r)^2}{r(r+1)}\\
       & = V\hat{C}_t + B - \gamma Y_{t}\hat{x}_{t}^1 -\sum_{r = 1}^{R} \frac{\gamma}{r}Q_{t}^{r}\hat{x}_{t}^{r} \\
       & + \sum\nolimits_i \left ( H_{i,t}(\Delta \hat{T}_{i,t}^c - \Delta T_{i,t}^d)\right) \\
       & \leq V\hat{C}_t + B + \sum\nolimits_i \left ( H_{i,t}(\Delta \hat{T}_{i,t}^c - \Delta T_{i,t}^d)\right).
    \end{align*}


We take expectations on both sides of the above equation.
\begin{subequations}
    \begin{align}
        & \qquad \mathbb{E}[\Delta (\mathbf{{\Theta}}_t^\ast)] + V\mathbb{E}[C_t^\ast]\\
        &\leq  V\mathbb{E}[\hat{C}_t] + B + \mathbb{E}\left[\sum\nolimits_i \left ( H_{i,t}(\Delta \hat{T}_{i,t}^c - \Delta T_{i,t}^d)\right) \right] \\
        & = V\mathbb{E}[\hat{C}_t] + B.
    \end{align}
\end{subequations}
    The last step stems from the independence of the optimal policy on queues and the stability of virtual queues.

Summing the above equation over all time slots $t \in \{1,2,..., T\}$, and dividing both two sides by $VT$, letting $T$ go to infinity, then we have:
\begin{equation}
    \begin{aligned}
       &\lim_{T \to \infty}\frac{1}{VT}\mathbb{E}[L(\mathbf{\Theta}_{T+1}^\ast)-L(\mathbf{\Theta}_{1}^\ast)] + \lim_{T \to \infty}\frac{1}{T}\sum_{t=1}^{T}\mathbb{E}[C_t^ \ast]\\
        & \leq \frac{B}{V}+ \lim_{T \to \infty}\frac{1}{T}\sum_{t=1}^{T}\mathbb{E}[\hat{C}_t^\ast].
    \end{aligned}
\end{equation}

Thus, we can obtain:
\begin{equation}
    {c}^\ast \leq  \hat{c}^\ast + \frac{B}{V},
\end{equation}
with 
\begin{align*}
        & B \stackrel{\triangle}{=}  \frac{\gamma}{2(R+1)}{Q_{max}^{R+1}}^2 + \frac{\gamma}{4}{Q_{max}^{1}}^2 +\gamma Y_{max}Q_{max}^1  \stepcounter{equation} \tag{\theequation}\\
       & + \sum_{r = 1}^{R} \frac{\gamma}{2(r+1)} ({a_{max}^r}^2 +2Q_{max}^{r+1}a_{max}^r)  +  \sum_{r = 1}^{R} \frac{\gamma}{2r} {x_{max}^{r}}^2\\
       & +  \frac{1}{2} \sum\nolimits_i \max [(\Delta T_{i,max}^c)^2,(\Delta T_{i,max}^d)^2] + \frac{\gamma}{2}\sum_{r=2}^R \frac{(Q_{\max}^r)^2}{r(r+1)}, 
\end{align*}
where $Q_{max}^{r}$, $Y_{max}$, $a_{max}^r$, $x_{max}^r$, $\Delta T_{i,max}^c$, and $\Delta T_{i,max}^d$ are the upper bounds of $Q_t^r$, $Y_t$, $a_t^r$, $x_t^r$, $ \Delta T_{i,t}^c$, and $ \Delta T_{i,t}^d$.

The proof is completed.

\end{appendices}

\end{document}